\documentclass[10pt]{amsart}
\usepackage[utf8]{inputenc}

\usepackage{amsmath}
\usepackage{amsthm}
\usepackage{amssymb}
\usepackage{datetime}
\usepackage{hyperref}
\usepackage{url}
\usepackage{mathrsfs}
\usepackage{graphicx}
\newdateformat{monthyeardate}{%
  \monthname[\THEMONTH], \THEYEAR}
  
\newtheorem{theorem}{Theorem}

\newtheorem{lemma}[]{Lemma}[section]
\newtheorem{conj}[]{Conjecture}

\newtheorem*{remark}{Remark}

\theoremstyle{definition}
\newtheorem{definition}{Definition}[section]
\newtheorem*{question}{Question}

\newcommand{\T}{\mathbb{T}}
\newcommand{\B}{\mathbb{B}}
\newcommand{\R}{\mathbb{R}}
\newcommand{\N}{\mathcal{N}}
\newcommand{\e}{\epsilon}

\begin{document}

\title{
    A Sum-Product Estimate for Well Spaced Sets
}

\author{Shengwen Gan}
\email{shengwen@mit.edu}
\address{Deparment of Mathematics, MIT, Cambridge, MA 02139}

\author{Alina Harbuzova}
\email{hadought@mit.edu}
\address{Deparment of Mathematics, MIT, Cambridge, MA 02139}

\date{\monthyeardate\today}

\begin{abstract}
    We study the $\delta$-discretized  sum-product estimates for well spaced sets. Our main result is: for a fixed $\alpha\in(1,\frac{3}{2}]$, we prove that for any $\sim|A|^{-1}$-separated set $A\subset[1,2]$ and $\delta=|A|^{-\alpha}$, we have: $\mathcal{N}(A+A, \delta)\cdot \mathcal{N}(AA, \delta) \gtrsim_{\epsilon}|A|\delta^{-1+\epsilon}$.
\end{abstract}
\maketitle
\section{Introduction}\label{introduction}
The sum-product problem was first proposed by Erd\"{o}s and Szemerédi. They conjectured:
\begin{conj}[Sum-Product Problem]\label{erdos-szemeredi}
Let $0<\e<1$. For any  $A\subset \R:$
$$\max{(|A+A|,|AA|)}\gtrsim |A|^{1+\epsilon}.$$
\end{conj}

The range for $\e$ has been improved by a lot of people. One improvement is made by Gy\"{o}rgy Elekes  \cite{elekes1997number}. He applied the Szemerédi-Trotter theorem and proved the conjecture \ref{erdos-szemeredi} for $\epsilon=\frac{1}{4}.$ In our paper, we will use the idea from Elekes' paper.\\

The $\delta$-discretized version of the sum-product problem was first discussed by  Katz and Tao \cite{katz2001some}. They study the $(\delta,\sigma)_1$-set which is an analog of an $\sigma$-dimensional set in $\R$. More precisely, we say $A$ is an $(\delta,\sigma)_1$-set, if $A$ is a $\delta$-separated subset of $[1,2]$, $|A|\sim \delta^{-\sigma}$ and $|A\cap I|\lesssim |I|^{\sigma}|A|$. The analogs of the sumsets and productsets are then $\N(A+A,\delta)$ and $\N(AA,\delta)$. (Here, we use $\N(B,\delta)$ to denote the maximal cardinality of the $\delta$-separated subset of $B$). The following question is conjectured:  
\begin{conj}
Let $0<\sigma<1$, then there exists a number $c=c(\sigma)>0$ such that for any $(\delta,\sigma)_1$-set $A\subset[1,2]$, we have 
$$\max{(\N(A+A,\delta),\N(AA,\delta))}\gtrsim \delta^{-c}|A|.$$
\end{conj}

Bourgain proved this conjecture in \cite{bourgain2010discretized}. Guth, Katz, and Zahl obtained an explicit bound that $c=\frac{\sigma(1-\sigma)}{4(7+3\sigma)}$ in \cite{guth2018discretized}. In this paper, we will consider sets with stronger spacing conditions: the set $A$ is roughly an arithmetic progression with uncertainty $|A|^{-1}$. Our result is the following:

\begin{theorem}\label{main}
Fix a number $\alpha \in (1,\frac{3}{2}]$. For any subset $A\subset [1,2]$ such that $A$ is $\sim|A|^{-1}$-separated, let $\delta = |A|^{-\alpha}$, a scale much smaller than the separation of the set. Then we have:
$$\mathcal{N}(A+A, \delta)\cdot \mathcal{N}(AA, \delta) \gtrsim_{\epsilon}|A|\delta^{-1+\epsilon}= |A|^{1+\alpha-\epsilon'},$$
where $\e>0$ and $\e'=\e\alpha$.

As an immediate corollary:
$$\max\{\mathcal{N}(A+A, \delta), \mathcal{N}(AA, \delta)\} \gtrsim_{\epsilon} \delta^{-1/2+\epsilon}|A|^{1/2}=|A|^{\frac{1+\alpha}{2} - \epsilon'}.$$
\end{theorem}

\begin{remark}
Our estimate for the product $\mathcal{N}(A+A, \delta)\cdot \mathcal{N}(AA, \delta)$ is the best possible for if we consider an arithmetic progression $A$, then
$$\mathcal{N}(A+A,\delta)\cdot \mathcal{N}(AA,\delta)\lesssim |A|\cdot \delta^{-1}.$$
Therefore, our bound on the product is tight.
\end{remark}

One thing to be noted is that the statement of our theorem is numerically the same to Garaev's result \cite{garaev2008sum} which is the sum-product estimate in the finite field setting. Let's first state Garaev's result:

\begin{theorem}[Garaev, \cite{garaev2008sum}]
In $\mathbb{F}_p$, if $A\subset \mathbb{F}_p$ with $|A|>p^{\frac{2}{3}}$, then
$$\max\{|A + A|, |AA|\} \gtrsim p^{1/2}|A|^{1/2}.$$
Also, for any integer $N\in [1,p]$, one can construct a subset $A\subset\mathbb{F}_p$ with $|A|=N$, such that
$$\max\{|A + A|, |AA|\} \lesssim p^{1/2}|A|^{1/2}.$$
\end{theorem}

Note that the estimate $\max\{\mathcal{N}(A+A, \delta), \mathcal{N}(AA, \delta)\}\gtrsim_{\epsilon} \delta^{-1/2+\epsilon}|A|^{1/2}$ that we proved meets with Garaev's estimate $\max\{|A + A|, |AA|\}\gtrsim p^{1/2}|A|^{1/2}$ (If we think of $\delta^{-1}=p$ as the ``cardinality" of the ambient space). Also, in Theorem \ref{main}, the assumption $\alpha\leq\frac{3}{2}$ implies $|A|=\delta^{-1/\alpha}\geq\delta^{\frac{2}{3}}=p^{\frac{2}{3}}$ which is just the assumption in Garaev's result.

However, we are not able to construct an example as in Garaev's result to show the estimate is sharp. The key idea of constructing the example in Garaev's result is that we can find an arithmetic progression and a geometric progression whose intersection is not too small. But under the well spacing condition in our paper, we can show that the intersection of an arithmetic progression and a geometric progression is always small. The discussion is in Section \ref{otherresults}.\\
 
Actually, we believe our estimate for $\max\{\mathcal{N}(A+A, \delta), \mathcal{N}(AA, \delta)\}$ is far from being optimal and we also think it is reasonable to conjecture that one of the sumset or the productset should have full size $\delta^{-1}$. We ask the following question: 

\begin{question}\label{mainconj}
Does the one of the sumset and the productset has full size?
More precisely, does there exists an $\alpha>1$ such that the following is true? 

For any subset $A\subset [1,2]$, with $|A| = N$ and $A$ is $\sim|A|^{-1}$-separated, we let $\delta = |A|^{-\alpha}$, then we have:
$$\max\{\mathcal{N}(A+A,\delta), \mathcal{N}(AA,\delta)\} \gtrsim_\e \delta^{-1+\e},$$
for any $\e>0$.
\end{question}

\vspace{5mm}
\noindent
\textbf{Ideas of the proof of Theorem 1.}
We will use Elekes' argument as in \cite{elekes1997number}, together with an $\delta$-discretized version of Szemerédi-Trotter theorem. 
In \cite{guth2019incidence}, Guth, Solomon and Wang proved an incidence estimate for well spaced tubes, which is exactly the $\delta$-discretized Szemerédi-Trotter theorem we want.

By repeating Elekes' argument, we boil down our theorem to proving an upper bound for $r$-rich $\delta$-tubes. In \cite{guth2019incidence}, an upper bound for $r$-rich $\delta$-balls is obtained. So by summing over all $r$'s, we obtain an upper bound for the incidences between $\delta$-balls and $\delta$-tubes. And finally by dividing the incidence by $r$, we obtain an upper bound for $r$-rich $\delta$-tubes.

\vspace{5mm}
\noindent
\textbf{Notations.}
We will use $|A|$ to denote the cardinality of the finite set $A$. We will use $A\lesssim B$ to denote that $A\leq CB$ for constant $C$ which depends only on the dimension $n.$ $A\sim B$ will mean $A\lesssim B$ and $B\lesssim A.$ We will use $A\lesssim_{\e} B$ to denote $A\leq C_{\e}B$ for some constant $C_\e$ depending on $\e$. For $A\subset\R$, we will use $\N(A,\delta)$ to denote the maximal cardinality among all $\delta$-separated subset of $A$.

\section{Proof of the Theorem \ref{main}}\label{statement}

\subsection{\texorpdfstring{$\delta$}{text}-discretized Szemerédi-Trotter Theorem}\label{stanalog}
In our proof we will need a variation of the Szemerédi-Trotter Theorem for $\delta$-discretized lines and points, namely the $\delta$-tubes and $\delta$-balls. Here, a $\delta$-tube is a $\delta\times 1$ rectangle and a $\delta$-ball is a ball of radius $\delta$.

To state the theorem from \cite{guth2019incidence}, we will first provide some notations.

\begin{definition}[Essentially distinct balls and tubes]
For a set of $\delta$-balls $\B$, we say these balls are essentially distinct if for any $B_1\neq B_2\in\B$, $|B_1\cap B_2|\leq (1/2)|B_1|$. Similarly, for a set of $\delta$-tubes $\T$, we say these tubes are essentially distinct if for any $T_1\neq T_2\in\T$, $|T_1\cap T_2|\leq (1/2)|T_1|$.
\end{definition}

In the rest of the paper, we will always consider essentially distinct $\delta$-balls and essentially distinct $\delta$-tubes.

\begin{definition}[Intersection of a $\delta$-ball and a $\delta$-tube]
We will say that a $\delta$-ball intersects a $\delta$-tube if the center of the $\delta$-ball lies inside the $\delta$-tube.
\end{definition}

\begin{definition}[$r$-rich balls and $r$-rich tubes]
Given a set of $\delta$-tubes $\mathbb{T}$, let's define the $r$-rich balls of $\T$ in the following way. We choose a set $\B$ to be a maximal set of essentially distinct $\delta$-balls (For example, $\B$ consists of all $\delta$-balls centered at $(\delta/2) \mathbb{Z}^2$, also note that the choice of $\B$ does not matter for our applications).
Let $P_r(\mathbb{T}):=\{ B\in\B: B\ \text{intersects more than}\ r\ \text{and less than}\ 2r\ \text{tubes from}\ \T \}$. We say $P_r(\T)$ is the set of $r$-rich $\delta$-balls. 

Similarly, given a set of $\delta$-balls $\B$, we define the $r$-rich tubes of $\B$ to be the set $P_r(\B):=\{ T\in\T: T\ \text{intersects more than}\ r\ \text{and less than}\ 2r\ \text{balls from}\ \B \}$. Here, $\T$ is a maximal set of essentially distinct $\delta$-tubes.
\end{definition}

From the definition, $P_r(\T)$ is basically a set of essentially distinct $\delta$-balls each of which intersects about $r$ tubes from $\T$, and $P_r(\B)$ is basically a set essentially distinct $\delta$-tubes each of which intersects about $r$ balls from $\B$.

With the definitions above, we can now state the following analog of the Szemerédi-Trotter Theorem proven in \cite{guth2019incidence}:

\begin{theorem}[Guth, Solomon, Wang, \cite{guth2019incidence}]\label{ST}
Suppose that $\delta\leq W^{-1}\leq 1.$ 
Suppose that $\mathbb{T}$ is a set of $\delta$-tubes in $[0, 1]^2$. We say $\T$ is $W^{-1}$-well spaced if $\lesssim 1$ $\delta$-tube of $\mathbb{T}$ lie in each $W^{-1}\times 1$ rectangle. 
$$ If \  r > \max (\delta^{1-\epsilon} W^{2}, 1),$$
$$then \  |P_r(\mathbb{T})| \lesssim_{\epsilon} \delta^{-\epsilon} r^{-3} W^{4}.$$
\end{theorem}

\begin{remark}
The original theorem requires $|\T|$ to have full size $\sim W^{2}$. Of course we can drop this requirement. To see this, we add some tubes to our $\T$ to get $\T'$, which still satisfies the spacing conditions and with $|\T'|\sim W^{2}$. We see $P_r(\T)\leq\sum_{s\geq r, s\ \rm{dyadic}}|P_s(\T')|\lesssim_{\epsilon} \delta^{-\epsilon} r^{-3} W^{4}.$  
\end{remark}

This theorem estimates the number of the $r$-rich $\delta$-balls. For our purposes, we will derive an estimate for the number of intersections between the tubes and balls, and as an immediate consequence, an estimate for the number of the $r$-rich $\delta$-tubes.

\begin{lemma}[Incidence estimates for $r$-rich balls]\label{ST2}
Suppose $\T$ is a set of $\delta$-tubes $\T$ satisfying the conditions in Theorem \ref{ST}, and $\B$ is a set of essentially distinct $\delta$-balls. Let $I(\mathbb{T},\B)$ be the number of the intersections between $\T$ and $\B$, then:
$$I(\mathbb{T},\B) \lesssim_{\epsilon} \delta^{1-\epsilon} |\B| W^{2} + \delta^{-2+\e},$$
and as a consequence:
$$P_r(\B)\lesssim_{\epsilon} \delta^{1-\e} |\B| W^{2} r^{-1} + \delta^{-2+\e} r^{-1}.$$
\end{lemma}
\begin{proof}
Recall that $P_r(\T)$ is the set of $\delta$-balls that intersect with at least $r$ tubes and less than $2r$ tubes. Then:
$$I(\mathbb{T}, \B) \lesssim \sum_{i=0}^{\log_2 W^{2}} |P_{2^i}(\T)| 2^i = \sum_{i=0}^{\log_2 \delta^{1-\e} W^{2}} |P_{2^i}(\T)| 2^i + \sum_{i=\log_2 \delta^{1-\e} W^{2}}^{\log_2 W^{2}} |P_{2^i}(\T)| 2^i.$$
For all $i$, we have $|P_{2^i}(\T)| \leq |\B|$ and therefore 
$$ \sum_{i=0}^{\log_2 \delta^{1-\e} W^{2}} |P_{2^i}(\T)| 2^i \leq |\B| \sum_{i=0}^{\log_2 \delta^{1-\e} W^{2}} 2^i \lesssim \delta^{1-\e} |\B| W^{2}.$$
Moreover by Theorem \ref{ST}, we have
$$\sum_{i=\log_2 \delta^{1-\e} W^{2}}^{\log_2 W^{2}} |P_{2^i}(\T)| 2^i \lesssim\sum_{i = \log_2 \delta^{1-\e} W^{2} }^{\log_2 W^{2}} W^{4} 2^{- 2 i}\delta^{-\epsilon} \lesssim W^{4} (\delta^{1-\e} W^{2})^{-2}\delta^{-\epsilon} = \delta^{-2+\epsilon}.$$
Combining these two inequalities, we get:
$$I(\mathbb{T},\B) \lesssim \delta^{1-\e} |\B| W^{2} + \delta^{-2+\e}.$$
Note that $r$- rich tubes contribute at most $r P_r(\B)$ to $I(\mathbb{T}, \B),$ and thus
$$P_r(\B) \lesssim \delta^{1-\epsilon} |\B| W^{2} r^{-1} + \delta^{-2+\epsilon} r^{-1}.$$
\end{proof}

\subsection{Reduce the set \texorpdfstring{$A$}{text} to the \texorpdfstring{$\delta$}{text}-lattice}\label{discr}
In the proof of Theorem \ref{main}, we can actually assume $A\subset \delta\mathbb{Z}$.

To see this, let $A$ be the set as in the Theorem \ref{main}. We will define another set $A'\subset \delta\mathbb{Z}$ as a replacement of $A$. 
$A'$ is obtained by replacing all the points of $A$ by their closest points in $\delta\mathbb{Z}$. More precisely, if $A=\{a_1,\cdots,a_k\}$ and let $b_i$ be the point in $\delta\mathbb{Z}$ which is closest to $a_i$, then $A'=\{b_1\cdots,b_k\}$. Under our assumption $\delta$ is much less than the separation of points in $A$, so $|A'|=|A|$ and $A'$ is also $\sim |A|^{-1}$-separated.

Consider the sumset $A'+A'$. For all $a_1', a_2' \in A'$ there exist corresponding $a_1, a_2 \in A,$ s.t. $a_1 + a_2 = a_1' + a_2' + O(\delta)$. Thus, $$\mathcal{N}(A+A, \delta) \gtrsim \mathcal{N}(A'+A',\delta).$$
Similar consideration gives $$\mathcal{N}(AA, \delta) \gtrsim \mathcal{N}(A'A',\delta).$$
Thus, in order to prove Theorem \ref{main}, it is sufficient to prove the following simplified theorem:
\begin{theorem}[simplified version]\label{main2}
Fix a number $\alpha\in(1,\frac{3}{2}]$.
For any $A\subset[1,2]\cap\delta\mathbb{Z}$ such that $|A|=N$ and A is $\sim N^{-1}$-separated, we let $\delta=N^{-\alpha}$. Then, $$\mathcal{N}(A+A, \delta)\cdot \mathcal{N}(AA, \delta) \gtrsim_{\e} N^{1+\alpha-\e}.$$
for any $\e>0$.
\end{theorem}

In the following subsection, we will prove this version of the theorem.

\subsection{Proof of the Theorem \ref{main2}}\label{mainproof}
The idea of the proof is roughly the same as in Elekes' work \cite{elekes1997number}.

Let $A = \{a_i, \ 1\leq i \leq N\}.$ First, consider $N^2$ line segments $l_{ij}=\{(x,y): y = a_j(x - a_i)\quad (0\leq x\leq 4)\}$, for any $a_i, a_k \in A.$ Let $\T:=\{\delta-\text{neighborhood of}\ l_{ij}\}$. We see that $\T$ is a set of $N^2$ $\delta$-tubes. Also, we can check $\T$ is $\sim N^{-1}$-well spaced. To see this, let's pick two different segments $l_{ij}$ and $l_{kl}$. If $i\neq k$, then their slopes are different and differed by at least $\sim N^{-1}$ due to the spacing condition of $A$. If $i=k$, then $j\neq l$ and so their intersections with $y$-axis have distance $|a_ia_j-a_ia_l|\gtrsim N^{-1}$. In both cases, we see that $l_{ij}$ and $l_{kl}$ can not lie in a same fat tube of size $\sim N^{-1}\times 4$.

Now define the set of our $\delta$-balls. Consider the set of points $(A+A)\times Q$, where $Q$ is the maximum possible $\delta$-separated subset of $AA$ (then $|Q| = \mathcal{N}(AA, \delta)$). Note that we have assumed $A$ to be a subset of $\delta\mathbb{Z}$, so $A+A\subset \delta\mathbb{Z}$ is $\delta$-separated and $|A+A|=\N(A+A,\delta) $. Now we define $\mathbb{B}$ to be the set of $\delta$-balls centered at $(A+A)\times Q$, so $|\mathbb{B}|=\N(A+A,\delta)\N(AA,\delta)$. We will first prove that all the tubes in $\T$ intersect at least $N$ balls from $\mathbb{B}$ (to apply later Lemma \ref{ST2} with $r = N$ and $W=N^{-1}$):

\begin{lemma}
Each tube in $\T$ intersects at least $N$ $\delta$-balls from $\mathbb{B}$.
\end{lemma}
\begin{proof}
Consider any tube and assume it is the $\delta$-neighborhood of the line segment $y=a_j(x-a_i) (0\leq x\leq 4)$.
Consider $N$ points in $(A+A)\times AA$: $\{(a_i+a_k,a_ja_k)\}_{1\leq k\leq N
}.$ All these points lie on the line segment. So, it suffices to find $N$ points from $(A+A)\times Q$ each of which is $\delta$-close to one of the aforementioned $N$ points. To do this, we first note that for every $a_ja_k$ there exists a $q_k\in Q$ such that $|q_k-a_ja_k|\leq \delta$ by the definition of $Q$. Also note that $|a_ja_k-a_ja_{k'}|\geq N^{-1}>>\delta$ for different $k,k'$, so $q_k$'s are different for different $k$'s. Therefore we find $N$ points $\{(a_i+a_k,q_k)\}_{1\leq k\leq N}$ from $(A+A)\times Q$ that are $\delta$-close to the line segment $y=a_j(x-a_i) (0\leq x\leq 4)$, and hence
the $\delta$-tube formed by this segment intersects at least $N$ $\delta$-balls from $\mathbb{B}$.
\end{proof}

Now we can apply Lemma \ref{ST2} to our set of essentially distinct tubes $\T$ and essentially distinct balls $\mathbb{B}$ with $W = N^{-1}.$ This gives us:

\begin{equation}\label{eq1}
N^2=|A|^2 < P_{N}(\mathbb{B}) \lesssim \delta^{1-\e}|\mathbb{B}|N^2N^{-1}+\delta^{2+\e}N^{-1} \lesssim N^{1-\alpha+\e} |\mathbb{B}| + N^{2\alpha -1-\epsilon'}
\end{equation}
$$\Rightarrow \N(A+A,\delta)\N(AA,\delta)=|\mathbb{B}| \gtrsim N^{1+\alpha-\e}.$$

Here we used the condition that $\alpha \le \frac{3}{2}.$

\newpage

\section{Other Results}\label{otherresults}
\subsection{The intersection of AP and GP}
We will show in some sense that the intersection of an arithmetic progression and a geometric progression is small.
\begin{theorem}\label{blablabla}
Let's fix an $\alpha\in[1,3/2]$. For any $N$, let $\delta=N^{-\alpha}$. Consider a length-$N$ arithmetic progression $A = \{1+iN^{-1}, 1\leq i \leq N\}$ and a                                            geometric progression $G = \{q^i, 1 \leq i\leq N\}$ with $q^N-1 \gtrsim 1$. Then $$ |G\cap \mathcal{E}_{\delta} (A)| \lesssim_\e N^{\max\{\alpha-1/2,(3-\alpha)/2\}+\e}$$ for any $\e>0$ $(\mathcal{E}_{\delta}(A)$ denotes the $\delta$-neighborhood of the set $A)$. 
\end{theorem}
\begin{proof}
Let $B = G\cap \mathcal{E}_{\delta} (A)$, and assume that $|B|\gtrsim N^{\max\{\alpha-1/2,(3-\alpha)/2\}+\e}.$ 
The condition $q^N-q = \Omega(1)$ implies that there cannot be two elements from $G$ that are in the $\delta$-neighborhood of the same element of $A.$ Actually, $q^N=1+\Omega(1)$ implies $q=(1+\Omega(1))^{1/N}>1+\frac{\Omega(1)}{N}$, and hence $q^{i+1}-q^i\geq q-1\gtrsim \frac{1}{N}$.

$A$ is an arithmetic progression, and therefore, $\mathcal{N}(A+A,\delta) \sim |A+A|\leq 2 N.$ Similarly $\mathcal{N}(G G, \delta) \sim |G G| \leq 2 N.$ Thus, because $B\subset \mathcal{E}_{\delta}(A), B\subset G,$ $$\mathcal{N}(B+B,\delta) \lesssim \mathcal{N}(A+A,\delta), \mathcal{N}(B B,\delta) \leq \mathcal{N}(G G,\delta),$$
and so $$\mathcal{N}(B+B,\delta)\cdot \mathcal{N}(B B, \delta) \lesssim N^2,$$
We will obtain a lower bound for $\mathcal{N}(B+B,\delta)\cdot \mathcal{N}(B B, \delta)$ to get a contradiction.

We put $A=B$ in Theorem \ref{main2}. We do not necessarily have $|B|\sim N$, but Equation (\ref{eq1}) still holds. Let's write down here:
$$ |B|^2 \lesssim \delta^{1-\e}MN^2N^{-1}+\delta^{2+\e}N^{-1} \lesssim N^{1-\alpha+\e} M + N^{2\alpha -1-\epsilon'} $$
Here, $M=\mathcal{N}(B+B,\delta)\cdot \mathcal{N}(B B, \delta)$. By our assumption, $|B|\gtrsim N^{\alpha-1/2}$, so we have
$$ |B|^2 \lesssim N^{1-\alpha+\e} M\lesssim N^{3-\alpha+\e}$$
which is a contradiction.
\end{proof}

\vspace{5mm}
\noindent
\textbf{Acknowledgement.}
We would like to thank Professor Larry Guth for suggesting this project and for many helpful discussions. Also, We would like to thank MIT UROP+ program for providing an opportunity to conduct this research.

\newpage

\bibliographystyle{alpha}
\bibliography{references}
\end{document}